\documentclass[11pt]{amsart}

\usepackage{hyperref}
\hypersetup{nesting=true,debug=true,naturalnames=true}
\usepackage{graphicx,amssymb,upref}
\hypersetup{
    colorlinks=true,
    citecolor=blue,
    linkcolor=blue,
    filecolor=blue,      
    urlcolor=blue,
}

\let\<\langle
\let\>\rangle
\usepackage[all,pdf]{xy}
\UseComputerModernTips

\let\uml\"

\title[A constructive proof]{A constructive proof that the Hanoi towers group has non-trivial rigid kernel} 

\author{Rachel Skipper}  
\address{Department of Mathematics \\
Binghamton University \\
P.O. Box 6000 \\
Binghamton, NY 13902-6000, USA} 
\email{skipper@math.binghamton.edu}  
\thanks{The author is grateful to Marcin Mazur for his helpful discussions in preparation of this paper.} 


\subjclass[2010]{}

\usepackage{mathabx, tikz}

\newcommand{\St}{\text{Stab}}
\newcommand{\Ri}{\text{Rist}}

\newtheorem{theorem}{Theorem}
\newtheorem{lemma}[theorem]{Lemma}
\newtheorem{corollary}[theorem]{Corollary}
\newtheorem{remark}[theorem]{Remark}

\usetikzlibrary{arrows}

\definecolor{wwqqzz}{rgb}{0.4,0.,0.6}
\definecolor{ttttff}{rgb}{0.2,0.2,1.}
\definecolor{qqwuqq}{rgb}{0.,0.39215686274509803,0.}
\definecolor{ffcctt}{rgb}{1.,0.8,0.2}
\definecolor{ffwwqq}{rgb}{1.,0.4,0.}
\definecolor{ffqqtt}{rgb}{1.,0.,0.2}
\definecolor{zzttqq}{rgb}{0.6,0.2,0.}

\begin{document} 
 
\begin{abstract} 
In 2012, Bartholdi, Siegenthaler, and Zalesskii computed the rigid kernel for the only known group for which it is non-trivial, the Hanoi towers group. There they determined the kernel was the Klein four-group. In this note, we present a simpler proof of this theorem. In the course of the proof, we also compute the rigid stabilizers and present proofs that this group is a self-similar, self-replicating, regular branch group.
\end{abstract}
\maketitle


\section{Introduction}

Since the construction of the first Grigorchuk group in 1980, the study of branch groups has developed into an important area in group theory. Branch groups derive their value from the unusual properties that groups in this class can exhibit. Amenable but not elementary amenable groups, groups of finite width, groups with intermediate growth, and finitely generated infinite torsion groups are a few of the types that can arise. As a result, these groups have been heavily studied in recent years \cite{BGS03}.

Showing that these groups have interesting properties and understanding why are equally important tasks as the latter can be used to gain a deeper understanding of these groups and eventually used to construct groups with additional noteworthy properties. For this reason, constructive proofs using the underlying geometry and properties of the group as opposed to more abstract techniques are essential.

In this note we give a short proof of the theorem of \cite{BSZ12}, namely that the rigid kernel for the Hanoi Towers Group is the Klein four-group. Along the way we establish various other properties of that group. Since this is the only branch group thus far shown to have non-trivial rigid kernel, acquiring a deeper understanding of the Hanoi towers group is desirable, a task we seek to achieve here.

In section 2, we describe various properties of branch groups as well as provide the necessary terminology and notation for the rest of the paper. In section 3, we present the congruence subgroup problem for branch groups. In section 4, we describe the ``Towers of Hanoi" game and present the Hanoi towers group which models this game for any number of disks. And finally, in section 5, we prove a number of properties of this group along with the main theorem.

\section{Branch groups}

Branch groups are defined in terms of their actions on trees, so for this reason we introduce some initial vocabulary and notation to aid in the discussion of groups of this type.

For any vertex $u$ in a rooted tree, its \textbf{level} will be defined as the length of the geodesic from the root to $u$ and denoted $|u|$. The tree is called \textbf{spherically homogeneous} if every vertex on a given level has the same valency.

An infinite, spherically homogeneous, rooted tree is fully determined by a sequence of integers $\overline{m}=(m_1, m_2, \dots )$ where each vertex of level $n-1$ has $m_n$ adjacent vertices of level $n$. $\mathcal{T}_{\overline{m}}$ will denote the tree with defining sequence $\overline{m}$. The tree is called \textbf{regular} when the defining sequence is constant, $d:=m_1=m_2=\cdots$. Such a tree is referred to as a $d$-ary tree. When the defining sequence is either arbitrary or clear from the context, the subscript will be dropped. With this notation, we will write $\emptyset$ for the root and we will identify a vertex $u$ of level $n$ with a sequence of integers $u=(u_1, u_2, \dots, u_n)$ where $1 \leq u_i \leq m_{i}$ and where the prefixes of the sequence correspond to the vertices on the geodesic between $u$ and $\emptyset$. Then the set of vertices of level n in $\mathcal{T}$ can be ordered linearly using the lexicographical ordering. Thus when convenient, we will number the vertices of level $n$ by the indexing set $\{1, 2, \dots, m_1\cdots m_n\}$.

When we wish to indicate that the tree is a $d$-ary tree, an alternate notation will be used. We will use $X$ for the set $\{1, 2, \dots, d\}$, $X^n$ for for sequences of length $n$ in $X$ (the vertices of level $n$ in the tree), and $X^*$ for the infinite $d$-ary tree.

$V(\mathcal{T}_{\overline{m}})$ will denote the set of vertices of $\mathcal{T}_{\overline{m}}$. The  \textbf{automorphism group} of $\mathcal{T}_{\overline{m}}$, denoted Aut$(\mathcal{T}_{\overline{m}})$, consists of bijections from $V(\mathcal{T}_{\overline{m}})$ to $V(\mathcal{T}_{\overline{m}})$ that preserve the root and preserve edge incidences. Thus, vertices of the same level in $\mathcal{T}_{\overline{m}}$ can only be permuted among themselves. Because of this, an element $g$ in Aut$(\mathcal{T}_{\overline{m}})$ can be regarded as a labeling of the vertices of $\mathcal{T}_{\overline{m}}$ by permutations, $\{g(v)\}_{v \in V(\mathcal{T}_{\overline{m}})}$, where if $|v|=n$ then $g(v)\in S_{m_n}$, the symmetric group on $m_n$ letters. Then for a vertex $u=(u_1, u_2, \dots, u_n)$, the action of $g$ is computed as

 $$u^g:=(u_1^{g(\emptyset)}, u_2^{g(u_1)}, \dots, u_n^{g(u_1, \dots, u_{n-1})}).$$
 
We say a vertex $v \in V(\mathcal{T}_{\overline{m}})$ is a \textbf{descendant} of $u$ if the geodesic from $v$ to $\varnothing$ includes the geodesic from $u$ to $\varnothing$. The set of descendants of $u$ forms the subtree rooted at $u$, denoted $\mathcal{T}_u$. If $\mathcal{T}_{\overline{m}}$ is a spherically homogeneous, rooted tree then for any $n$, each subtree of $\mathcal{T}_{\overline{m}}$ rooted at a vertex of level $n$ is canonically isomorphic to $\mathcal{T}_{\psi^n(\overline{m})}$, where $\psi^n(\overline{m})=(m_{n+1}, m_{n+2}, \dots)$. As a result, there is a natural isomorphism Aut$(\mathcal{T}_{\overline{m}})\cong \text{Aut}(\mathcal{T}_{\psi^n(\overline{m})})\wr ~ M_n= (\prod \text{Aut}(\mathcal{T}_{\psi^n(\overline{m})})) \rtimes M_n$ where $M_n=(\cdots (S_{m_n}\wr S_{m_{n-1}})\wr \cdots) \wr S_{m_1}$ with $S_k$ signifying the symmetric group on $k$ letters. The iterated wreath product $M_n$ is the automorphism group of the finite subtree of $\mathcal{T}_{\overline{m}}$ consisting of vertices of level less than or equal to $n$. When the tree is a regular $d$-ary tree, then the subtree rooted at any vertex is isomorphic to the full tree and we get the natural isomorphism Aut$(X^*)\cong \text{Aut}(X^*)\wr M_n$ with $M_n={(\cdots (S_{d}\wr S_{d})\wr \cdots) \wr S_{d}}$. 

Following the notation of \cite{BSZ12}, for $g\in$Aut$(\mathcal{T}_{\overline{m}})$ and for $u$ a vertex of level $n$, we will denote by $g@u$ the $u$th coordinate of $g$ in the canonical identification Aut$(\mathcal{T}_{\overline{m}})\cong (\text{Aut}(\mathcal{T}_{\psi^n(\overline{m})})\times \cdots \text{Aut}(\mathcal{T}_{\psi^n(\overline{m})}))\wr ~ M_n$, and we will call it the \textbf{state} of $g$ at $u$. Any element $g\in$Aut$(\mathcal{T}_{\overline{m}})$, can be described as $g=(g_1, \dots, g_{m_1})\sigma$ where $\sigma$ the permutation labeling at the root and $g_i$, $1\leq i \leq m_1$, is the state of $g$ at the $i$th subtree rooted at the first level. In the case of the $d$-ary tree, for $g\in \text{Aut}(X^*)$, the states of $g$ are also in Aut($X^*$). 

For any subgroup $G$ of Aut$(\mathcal{T}_{\overline{m}})$, four families of subgroups arise naturally. For a vertex $u \in V(\mathcal{T}_{\overline{m}})$, the \textbf{vertex stabilizer} of $u$, $\St_G(u)$, is the set of elements in $G$ which fix the vertex $u$. In terms of the labeling of the vertices by elements in a symmetric group, this consists of the elements that necessarily have trivial labeling on all vertices on the path between $u$ and $\varnothing$, except possibly at $u$. For a non-negative integer $n$, the \textbf{$n$th level stabilizer}, $\St_G(n)$, is the normal subgroup $\displaystyle\bigcap_{|u|=n}\St_G(u)$. In terms of the labelings, this consists of the elements of $G$ with trivial labeling on all vertices $v$ where $|v|\leq n-1$. Note that for all $n$, $\St_G(n)$ has finite index in $G$. 

When $g\in \St_G(n)$, it can be defined by $g=(g_1, \dots, g_{m_1\cdots m_n})_n$ where each $g_i$, $1 \leq i \leq m_1\cdots m_n$, describes the state of $g$ at the $i$th subtree rooted at the $n$th level. In addition, we can canonically identify $\St_G(n)/\St_G(n+1)$ with a subgroup of $(S_{m_{n+1}})^{m_1\cdots m_n}$. Hence elements in this quotient will be described by a $m_1\cdots m_n$-tuple of permutations.

Next, the \textbf{rigid stabilizer} of a vertex $u$, $\Ri_G(u)$, consists of the elements of $G$ which act trivially outside of the subtree rooted at $u$. In terms of the labeling, this consists of elements that have trivial labeling on all vertices outside of $\mathcal{T}_u$. If $G$ acts transitively on all the levels of $\mathcal{T}_{\overline{m}}$, then for any two vertices $u$ and $v$ such that $|u|=|v|$, $\Ri_G(u)\cong \Ri_G(v)$ (and in fact are conjugate in $G$). And finally, for a non-negative integer $n$, the \textbf{$n$th level rigid stabilizer} is the normal subgroup $\Ri_G(n)={\langle \Ri_G(u)| \text{ } |u|=n \rangle} = {\displaystyle\prod_{|u|=n} \Ri_G(u)}$, the internal direct product of the rigid stabilizers of the vertices of level $n$. For any group $G$ acting faithfully on $\mathcal{T}_{\overline{m}}$, $\Ri_G(n)\leq \St_G(n)$.

A subgroup $G$ of Aut($\mathcal{T})$ is said to be a \textbf{branch group} if $G$ acts transitively on each level of $\mathcal{T}$ and for all $n$, $\Ri_G(n)$ has finite index in $G$. 

\section{The congruence subgroup problem for branch groups} 

The congruence subgroup property for branch groups derives its name from the congruence subgroup problem for $SL(n,\mathbb{Z})$ which asks if every subgroup of finite index in $SL(n, \mathbb{Z})$ contains a principal congruence subgroup, the kernel of the map $SL(n, \mathbb{Z}) \rightarrow SL(n, \mathbb{Z}/m\mathbb{Z})$ for some $m$. This is false for $n=2$ but was answered affirmatively for $n \geq 3$ in \cite{BLS64}. Similarly, we say that a branch group $G$ has the \textbf{congruence subgroup property} if every subgroup of finite index contains the $n$th level stabilizer for some $n$.  We can restate this in terms of profinite completions as follows. 

Since $\St_G(n)$ has finite index in $G$ for all $n$ and since this collection forms a descending collection of normal subgroups, taking ${\{\St_G(n)| n \in \mathbb{N} \}}$ as a basis for the neighborhoods of $\{1\}$ produces a profinite topology on $G$ (see Section 3.1 \cite{RZ10}) called the \textbf{congruence topology}. Likewise $\Ri_G(n)$ has finite index for all $n$, and in the same way produces a profinite topology called the \textbf{branch topology}. Further, $G$ has a third natural topology, the \textbf{full profinite topology} where $\mathcal{N}=\{N\unlhd G \mid |G:N|<\infty\}$ is taken as a basis for the neighborhoods of $\{1\}$. Observe that the congruence topology is weaker than the branch topology which is weaker than the full profinite topology. We can complete $G$ in terms of these topologies and obtain three profinite groups:

\begin{align*}
&\overline{G}=\underset{n \geq 1}{\varprojlim} G/\St_G(n)  &\textbf{the congruence completion}  \\
&\widetilde{G}=\underset{n\geq 1}{\varprojlim} G/\Ri_G(n)  &\textbf{the branch completion}  \\
&\widehat{G}=\underset{N \in \mathcal{N}}{\varprojlim}G/N  &\textbf{the profinite completion}  \\
\end{align*}

As $G$ is a subgroup of Aut$(\mathcal{T})$, we see $\bigcap_{n \geq 1}\St_G(n)=\{1\}$, G is residually finite and embeds into $\overline{G}$, $\widetilde{G}$, and $\widehat{G}$.

Thus $G$ has the congruence subgroup property if and only if the \textbf{congruence kernel}, $\ker(\widehat{G}\rightarrow \overline{G})$, is trivial. The \textbf{congruence subgroup problem} for branch groups asks not only whether a branch group has the congruence subgroup property but also to quantitatively describe the congruence kernel. Since there is a third topology at play, namely the branch topology, we can instead study two pieces of the congruence kernel, namely the \textbf{branch kernel}, $\ker(\widehat{G}\rightarrow \widetilde{G})$, and the \textbf{rigid kernel}, $\ker(\widetilde{G} \rightarrow \overline{G})$. Although a group may have many realizations as a branch group, each of these kernels are invariants of the group and are not dependent on the choice of realization \cite{Gar14}.

Many of the most studied branch groups have been shown to have a trivial congruence kernel, including the Fabrykowsky-Gupta group and the Gupta-Sidki group \cite{BGS03}, the Grigorchuk group and an infinite family of generalizations of the Fabrykowsky-Gupta group \cite{Gri00}, and GGS-groups with non-constant accompanying vectors \cite{Per07}, \cite{FGU17}.

Pervova \cite{Per07} constructed the first branch groups without the congruence subgroup property. Nevertheless, the groups in her infinite family, periodic EGS groups with non-symmetric accompanying vector, have non-trivial branch kernel but trivial rigid kernel. Likewise, the twisted twin of the Grigorchuk group was found to have non-trivial branch kernel but trivial rigid kernel \cite{BS09}.

Despite the existence of infinite families of groups having either trivial branch and trivial rigid kernel or non-trivial branch kernel but trivial rigid kernel, only one group appearing in the literature has been shown to have non-trivial rigid kernel. It is the Hanoi towers group on three pegs \cite{BSZ12}. For this reason, we explore the various properties of the Hanoi towers group.

\begin{remark}
Since the writing of this paper, the author has constructed new examples of branch groups with non-trivial rigid kernel. They appear as finite index subgroups of a generalization of the Hanoi towers group \cite{Ski17}.
\end{remark}

\section{The game and the group}
The Hanoi towers group $\Gamma$ was first introduced by Grigorchuk and \u{S}uni\'{k} \cite{GS06}. The action of $\Gamma$ on the first $n$ levels of the tree models the ``Towers of Hanoi'' game with $n$ disks, hence the name. We start by describing the game. 

The ``Towers of Hanoi'' game for three pegs and $n$ disks works as follows. It begins with 3 pegs and $n$ disks each of varying size organized from largest to smallest on the first peg. Figure 1 shows this initial game state for $n=6$. The goal of the game is to move each of the disks from the first peg to the third peg through a series of moves. Each move consists of taking the top disk from one peg and placing it atop another peg with the restriction that at no point can a disk be on top of a smaller disk.

\begin{figure}[ht]
\begin{center}
\begin{tikzpicture}[line cap=roundline join=round,>=triangle 45,x=1.0cm,y=1.0cm, scale=.8];
\fill[line width=2.pt,color=zzttqq,fill=zzttqq,fill opacity=0.1] (-0.54,-2.52) -- (-0.54,-3.27) -- (2.46,-3.27) -- (2.46,-2.52) -- cycle;
\fill[line width=2.pt,color=zzttqq,fill=zzttqq,fill opacity=0.1] (4.46,-2.52) -- (4.46,-3.27) -- (7.46,-3.27) -- (7.46,-2.52) -- cycle;
\fill[line width=2.pt,color=zzttqq,fill=zzttqq,fill opacity=0.1] (9.46,-2.52) -- (9.46,-3.27) -- (12.46,-3.27) -- (12.46,-2.52) -- cycle;
\fill[color=ffqqtt,fill=ffqqtt,fill opacity=1.0] (-0.34,-2.22) -- (-0.34,-2.52) -- (2.26,-2.52) -- (2.26,-2.22) -- cycle;
\fill[color=ffwwqq,fill=ffwwqq,fill opacity=1.0] (-0.14,-1.92) -- (-0.14,-2.22) -- (2.06,-2.22) -- (2.06,-1.92) -- cycle;
\fill[color=ffcctt,fill=ffcctt,fill opacity=1.0] (0.04,-1.62) -- (0.04,-1.92) -- (1.86,-1.92) -- (1.86,-1.62) -- cycle;
\fill[color=qqwuqq,fill=qqwuqq,fill opacity=1.0] (0.24,-1.3) -- (0.24,-1.62) -- (1.66,-1.62) -- (1.66,-1.3) -- cycle;
\fill[color=ttttff,fill=ttttff,fill opacity=1.0] (0.44,-1.) -- (0.44,-1.3) -- (1.46,-1.3) -- (1.46,-1.) -- cycle;
\fill[color=wwqqzz,fill=wwqqzz,fill opacity=1.0] (0.64,-0.7) -- (0.64,-1.) -- (1.26,-1.) -- (1.26,-0.7) -- cycle;
\draw [line width=2.pt,color=zzttqq, line cap=round] (-0.54,-2.52)-- (-0.54,-3.27);
\draw [line width=2.pt,color=zzttqq, line cap=round] (-0.54,-3.27)-- (2.46,-3.27);
\draw [line width=2.pt,color=zzttqq, line cap=round] (2.46,-3.27)-- (2.46,-2.52);
\draw [line width=2.pt,color=zzttqq, line cap=round] (2.46,-2.52)-- (-0.54,-2.52);
\draw [line width=2.pt,color=zzttqq, line cap=round] (4.46,-2.52)-- (4.46,-3.27);
\draw [line width=2.pt,color=zzttqq, line cap=round] (4.46,-3.27)-- (7.46,-3.27);
\draw [line width=2.pt,color=zzttqq, line cap=round] (7.46,-3.27)-- (7.46,-2.52);
\draw [line width=2.pt,color=zzttqq, line cap=round] (7.46,-2.52)-- (4.46,-2.52);
\draw [line width=2.pt,color=zzttqq, line cap=round] (9.46,-2.52)-- (9.46,-3.27);
\draw [line width=2.pt,color=zzttqq, line cap=round] (9.46,-3.27)-- (12.46,-3.27);
\draw [line width=2.pt,color=zzttqq, line cap=round] (12.46,-3.27)-- (12.46,-2.52);
\draw [line width=2.pt,color=zzttqq, line cap=round] (12.46,-2.52)-- (9.46,-2.52);
\draw [line width=3.2pt] (5.96,0.25)-- (5.96,-2.52);
\draw [line width=3.2pt] (10.96,0.25)-- (10.96,-2.52);
\draw [line width=3.2pt] (0.95,0.25)-- (0.95,-0.7);
\end{tikzpicture}
\end{center}
\caption{The beginning game state for the ``Towers of Hanoi''.}
\label{fig:startstate}
\end{figure}
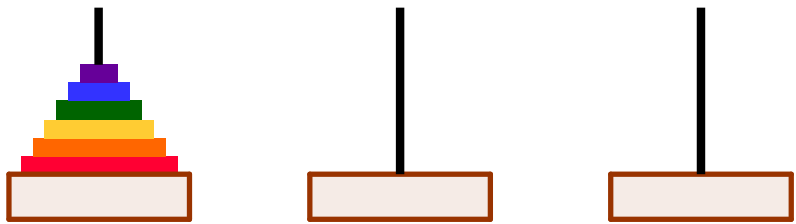

The restriction on the moves in the game limits a player's options to three possibilities. The first move, which will be called move $a$, transfers the smallest disk on pegs 2 and 3 between them. Likewise, move $b$ transfers the smallest disk on pegs 1 and 3 between them and move $c$ transfers the smallest disk on pegs 1 and 2 between them.

Any sequence of moves yields a game state which consists of the disks distributed across the three pegs such that on each peg, starting at the bottom and working up, the disks decrease in size. Thus, every game state in the $n$-disk game can be encoded as a sequence of $n$ integers between 1 and 3 in the following way: the first integer indicates the location of the smallest disk, the second integer indicates the location of the next smallest disk and so forth until the final integer indicates the location of the largest disk. For example, the Figure 2 shows a possible game state for the 6-disk game corresponding to the sequence $(2,1,3,2,2,1)$. 

\begin{figure}[ht]
\begin{center}
\begin{tikzpicture}[line join=round,>=triangle 45,x=1.0cm,y=1.0cm, scale=.8];
\fill[line width=2.pt,color=zzttqq,fill=zzttqq,fill opacity=0.1] (-0.54,-2.52) -- (-0.54,-3.27) -- (2.46,-3.27) -- (2.46,-2.52) -- cycle;
\fill[line width=2.pt,color=zzttqq,fill=zzttqq,fill opacity=0.1] (4.46,-2.52) -- (4.46,-3.27) -- (7.46,-3.27) -- (7.46,-2.52) -- cycle;
\fill[line width=2.pt,color=zzttqq,fill=zzttqq,fill opacity=0.1] (9.46,-2.52) -- (9.46,-3.27) -- (12.46,-3.27) -- (12.46,-2.52) -- cycle;
\fill[color=ffqqtt,fill=ffqqtt,fill opacity=1.0] (-0.34,-2.22) -- (-0.34,-2.52) -- (2.26,-2.52) -- (2.26,-2.22) -- cycle;
\fill[color=ffwwqq,fill=ffwwqq,fill opacity=1.0] (4.86,-2.22) -- (4.86,-2.52) -- (7.06,-2.52) -- (7.06,-2.22) -- cycle;
\fill[color=ffcctt,fill=ffcctt,fill opacity=1.0] (5.06,-1.92) -- (5.06,-2.22) -- (6.86,-2.22) -- (6.86,-1.92) -- cycle;
\fill[color=qqwuqq,fill=qqwuqq,fill opacity=1.0] (10.26,-2.22) -- (10.26,-2.52) -- (11.66,-2.52) -- (11.66,-2.22) -- cycle;
\fill[color=ttttff,fill=ttttff,fill opacity=1.0] (0.44,-1.92) -- (0.44,-2.22) -- (1.46,-2.22) -- (1.46,-1.92) -- cycle;
\fill[color=wwqqzz,fill=wwqqzz,fill opacity=1.0] (5.66,-1.62) -- (5.66,-1.92) -- (6.26,-1.92) -- (6.26,-1.62) -- cycle;
\draw [line width=2.pt,color=zzttqq, line cap=round] (-0.54,-2.52)-- (-0.54,-3.27);
\draw [line width=2.pt,color=zzttqq, line cap=round] (-0.54,-3.27)-- (2.46,-3.27);
\draw [line width=2.pt,color=zzttqq, line cap=round] (2.46,-3.27)-- (2.46,-2.52);
\draw [line width=2.pt,color=zzttqq, line cap=round] (2.46,-2.52)-- (-0.54,-2.52);
\draw [line width=2.pt,color=zzttqq, line cap=round] (4.46,-2.52)-- (4.46,-3.27);
\draw [line width=2.pt,color=zzttqq, line cap=round] (4.46,-3.27)-- (7.46,-3.27);
\draw [line width=2.pt,color=zzttqq, line cap=round] (7.46,-3.27)-- (7.46,-2.52);
\draw [line width=2.pt,color=zzttqq, line cap=round] (7.46,-2.52)-- (4.46,-2.52);
\draw [line width=2.pt,color=zzttqq, line cap=round] (9.46,-2.52)-- (9.46,-3.27);
\draw [line width=2.pt,color=zzttqq, line cap=round] (9.46,-3.27)-- (12.46,-3.27);
\draw [line width=2.pt,color=zzttqq, line cap=round] (12.46,-3.27)-- (12.46,-2.52);
\draw [line width=2.pt,color=zzttqq, line cap=round] (12.46,-2.52)-- (9.46,-2.52);
\draw [line width=3.2pt] (5.96,0.25)-- (5.96,-1.62);
\draw [line width=3.2pt] (10.96,0.25)-- (10.96,-2.22);
\draw [line width=3.2pt] (0.95,0.25)-- (0.95,-1.92);
\end{tikzpicture}
\end{center}
\caption{The game state corresponding to  $(2,1,3,2,2,1)$.}
\label{fig:gamestate}
\end{figure}
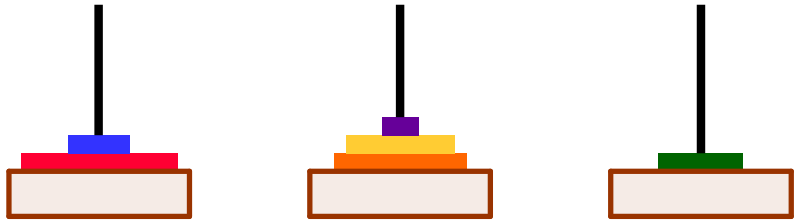

Integer sequences of length $n$ where the integers are between 1 and 3 can also be thought of as a vertex on the $n$th level in a rooted ternary tree as described in Section \ref{sec:branchgroups} and as seen in Figure 3.

\begin{figure}[ht]
\begin{center}
\begin{tikzpicture}[line cap=round,line join=round,>=triangle 45,x=1.0cm,y=1.0cm, scale=.4];
\draw  (0.,0.)-- (-9.,-2.);
\draw [line width=2pt](0.,0.)-- (0.,-2.);
\draw (0.,0.)-- (9.,-2.);
\draw (-9.,-2.)-- (-12.,-4.);
\draw  (-9.,-2.)-- (-9.,-4.);
\draw (-9.,-2.)-- (-6.,-4.);
\draw [line width=2pt](0.,-2.)-- (-3.,-4.);
\draw (0.,-2.)-- (0.,-4.);
\draw (0.,-2.)-- (3.,-4.);
\draw (9.,-2.)-- (6.,-4.);
\draw (9.,-2.)-- (9.,-4.);
\draw (9.,-2.)-- (12.,-4.);
\draw (-12.,-4.)-- (-13.,-6.);
\draw (-12.,-4.)-- (-12.,-6.);
\draw (-12.,-4.)-- (-11.,-6.);
\draw (-9.,-4.)-- (-10.,-6.);
\draw (-9.,-4.)-- (-9.,-6.);
\draw (-9.,-4.)-- (-8.,-6.);
\draw (-6.,-4.)-- (-7.,-6.);
\draw (-6.,-4.)-- (-6.,-6.);
\draw (-6.,-4.)-- (-5.,-6.);
\draw (-3.,-4.)-- (-4.,-6.);
\draw (-3.,-4.)-- (-3.,-6.);
\draw [line width=2pt](-3.,-4.)-- (-2.,-6.);
\draw (0.,-4.)-- (-1.,-6.);
\draw (0.,-4.)-- (0.,-6.);
\draw (0.,-4.)-- (1.,-6.);
\draw (3.,-4.)-- (2.,-6.);
\draw (3.,-4.)-- (3.,-6.);
\draw (3.,-4.)-- (4.,-6.);
\draw (6.,-4.)-- (5.,-6.);
\draw (6.,-4.)-- (6.,-6.);
\draw (6.,-4.)-- (7.,-6.);
\draw (9.,-4.)-- (8.,-6.);
\draw (9.,-4.)-- (9.,-6.);
\draw (9.,-4.)-- (10.,-6.);
\draw (12.,-4.)-- (11.,-6.);
\draw (12.,-4.)-- (12.,-6.);
\draw (12.,-4.)-- (13.,-6.);
\begin{tiny}
\draw[color=black] (0.14,0.36) node {$\varnothing$};
\draw[color=black] (-8.22,-2.19) node {$1$};
\draw[color=black] (0.6,-2.09) node {$2$};
\draw[color=black] (9.54,-2.09) node {$3$};
\draw[color=black] (-11.42,-3.99) node {$1$};
\draw[color=black] (-8.68,-3.88) node {$2$};
\draw[color=black] (-5.7,-3.9) node {$3$};
\draw[color=black] (-2.6,-4.04) node {$1$};
\draw[color=black] (0.38,-3.86) node {$2$};
\draw[color=black] (3.34,-3.88) node {$3$};
\draw[color=black] (6.46,-4.) node {$1$};
\draw[color=black] (9.28,-3.82) node {$2$};
\draw[color=black] (12.36,-3.92) node {$3$};
\draw[color=black] (-13.06,-6.42) node {$1$};
\draw[color=black] (-11.98,-6.42) node {$2$};
\draw[color=black] (-11.,-6.42) node {$3$};
\draw[color=black] (-10.02,-6.42) node {$1$};
\draw[color=black] (-9.06,-6.42) node {$2$};
\draw[color=black] (-7.98,-6.42) node {$3$};
\draw[color=black] (-6.98,-6.42) node {$1$};
\draw[color=black] (-5.98,-6.42) node {$2$};
\draw[color=black] (-4.98,-6.42) node {$3$};
\draw[color=black] (-4.04,-6.42) node {$1$};
\draw[color=black] (-2.96,-6.42) node {$2$};
\draw[color=black] (-1.96,-6.42) node {$3$};
\draw[color=black] (-0.96,-6.42) node {$1$};
\draw[color=black] (0.,-6.42) node {$2$};
\draw[color=black] (1.02,-6.42) node {$3$};
\draw[color=black] (2.,-6.42) node {$1$};
\draw[color=black] (3.02,-6.42) node {$2$};
\draw[color=black] (4.,-6.42) node {$3$};
\draw[color=black] (5.02,-6.42) node {$1$};
\draw[color=black] (6.08,-6.42) node {$2$};
\draw[color=black] (7.06,-6.42) node {$3$};
\draw[color=black] (8.04,-6.42) node {$1$};
\draw[color=black] (9.06,-6.42) node {$2$};
\draw[color=black] (10.04,-6.42) node {$3$};
\draw[color=black] (11.,-6.42) node {$1$};
\draw[color=black] (12.04,-6.42) node {$2$};
\draw[color=black] (13.16,-6.42) node {$3$};
\draw[thick, dotted] (-9, -7) -- (-9, -7.5);
\draw[thick, dotted] (0, -7) -- (0, -7.5);
\draw[thick, dotted] (9, -7) -- (9, -7.5);
\end{tiny}
\end{tikzpicture}
\end{center}
\caption{The rooted ternary tree}
\end{figure}
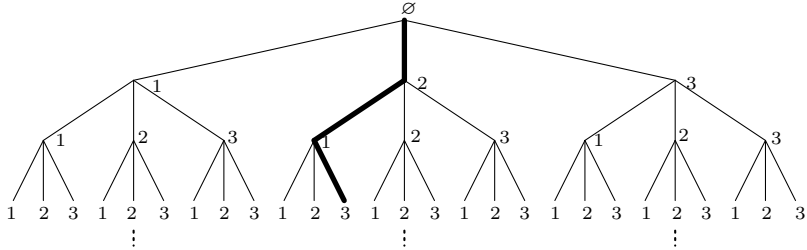

Since any move in the game takes one game state to another game state, i.e. takes one vertex on the $n$th level in the tree to another vertex on the $n$th level, each move can be thought of as an automorphism of the rooted ternary tree. Move $a$ should search for the first time a 2 or 3 appears in the path, and then switch them. Moves $b$ and $c$ should act similarly but instead with the numbers 1 and 3 and the numbers 1 and 2 respectively. For example, move $b$ takes the sequence $(2,1,3,2,2,1)$ to $(2,3,3,2,2,1)$.

In the same way we can define elements $a$, $b$ and $c$ acting on the whole ternary tree $X^*$. They are as follows:
$$a:=(a,1,1)\sigma_{23} \quad b:=(1,b,1)\sigma_{13} \quad c:=(1,1,c)\sigma_{12}$$
where we are using the isomorphism Aut$(X^*)\cong \text{Aut}(X^*)\wr S_3$.

\begin{figure}[ht]
\begin{center}
\begin{tikzpicture}[line cap=round,line join=round,>=triangle 45,x=1.0cm,y=1.0cm, scale=.42];
\draw (0.,0.)-- (-2.,-2.);
\draw (0.,0.)-- (0.,-2.);
\draw (0.,0.)-- (2.,-2.);
\draw (-2.,-2.)-- (-4.,-4.);
\draw (-2.,-2.)-- (-2.,-4.);
\draw (-2.,-2.)-- (0.,-4.);
\draw (-4.,-4.)-- (-6.,-6.);
\draw (-4.,-4.)-- (-4.,-6.);
\draw (-4.,-4.)-- (-2.,-6.);
\draw (6.,0.)-- (4.,-2.);
\draw (6.,0.)-- (6.04,-1.86);
\draw (6.,0.)-- (8.,-2.);
\draw (6.04,-1.86)-- (6.,-4.);
\draw (6.04,-1.86)-- (8.,-4.);
\draw (6.,-4.)-- (4.,-6.);
\draw (6.,-4.)-- (6.,-6.);
\draw (6.,-4.)-- (8.,-6.);
\draw (6.04,-1.86)-- (4.,-4.);
\draw (10.,-2.)-- (12.,0.);
\draw (12.,0.)-- (12.,-2.);
\draw (12.,0.)-- (14.,-2.);
\draw (14.,-2.)-- (12.,-4.);
\draw (14.,-2.)-- (14.,-4.);
\draw (14.,-2.)-- (16.,-4.);
\draw (16.,-4.)-- (14.,-6.);
\draw (16.,-4.)-- (16.,-6.);
\draw (16.,-4.)-- (18.,-6.);
\draw [thick, dotted] (2.,-2.3) -- (2.,-2.8);
\draw [thick, dotted] (0.,-4.3) -- (0.,-4.8);
\draw [thick, dotted] (-2.,-6.3) -- (-2.,-6.8);
\draw [thick, dotted] (6.,-6.3) -- (6.,-6.8);
\draw [thick, dotted] (10.,-2.3) -- (10.,-2.8);
\draw [thick, dotted] (12.,-4.3) -- (12.,-4.8);
\draw [thick, dotted] (14.,-6.3) -- (14.,-6.8);
\begin{scriptsize}
\draw[color=black] (0.14,0.36) node {$\sigma_{23}$};
\draw[color=black] (-1.30,-1.92) node {$\sigma_{23}$};
\draw[color=black] (0.3,-1.88) node {$1$};
\draw[color=black] (2.34,-1.88) node {$1$};
\draw[color=black] (-3.24,-3.82) node {$\sigma_{23}$};
\draw[color=black] (-1.66,-3.9) node {$1$};
\draw[color=black] (0.32,-3.9) node {$1$};
\draw[color=black] (-5.28,-5.94) node {$\sigma_{23}$};
\draw[color=black] (-3.68,-5.9) node {$1$};
\draw[color=black] (-1.64,-5.86) node {$1$};
\draw[color=black] (6.14,0.36) node {$\sigma_{13}$};
\draw[color=black] (4.36,-2.) node {$1$};
\draw[color=black] (6.7,-1.78) node {$\sigma_{13}$};
\draw[color=black] (8.34,-1.92) node {$1$};
\draw[color=black] (4.28,-3.96) node {$1$};
\draw[color=black] (6.7,-3.9) node {$\sigma_{13}$};
\draw[color=black] (8.42,-3.98) node {$1$};
\draw[color=black] (4.36,-5.88) node {$1$};
\draw[color=black] (6.7,-5.84) node {$\sigma_{13}$};
\draw[color=black] (8.38,-5.9) node {$1$};
\draw[color=black] (10.38,-1.94) node {$1$};
\draw[color=black] (12.14,0.36) node {$\sigma_{12}$};
\draw[color=black] (12.44,-1.9) node {$1$};
\draw[color=black] (14.68,-1.96) node {$\sigma_{12}$};
\draw[color=black] (12.68,-3.84) node {$1$};
\draw[color=black] (14.52,-3.84) node {$1$};
\draw[color=black] (16.8,-4) node {$\sigma_{12}$};
\draw[color=black] (14.62,-5.92) node {$1$};
\draw[color=black] (16.44,-5.84) node {$1$};
\draw[color=black] (18.6,-6.1) node {$\sigma_{12}$};
\end{scriptsize}
\end{tikzpicture}
\end{center}
\caption{The generators $a$, $b$, and $c$ of the Hanoi towers group}
\end{figure}
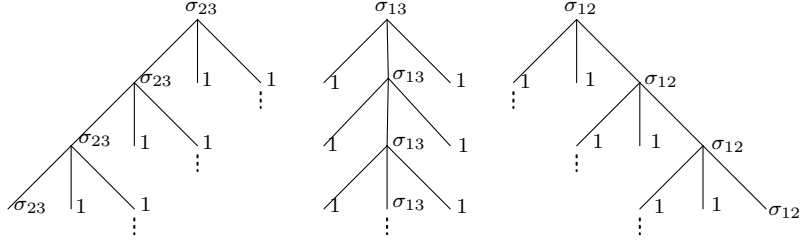

Figure 4 shows the labeling of the vertices by elements in $S_3$ for $a$, $b$, and $c$ respectively. Then the Hanoi towers group is $\Gamma:=~\langle a,b,c \rangle$. In \cite{BSZ12} a full presentation for $\Gamma$ is obtained. It is:

\begin{equation}
\label{presentation}
\Gamma=\langle a,b,c | a^2, b^2, c^2, \tau^n(w_1), \tau^n(w_2), \tau^n(w_3), \tau^n(w_4) \text{ for all } n\geq0 \rangle
\end{equation}

where $\tau$ is an endomorphism of $\Gamma$ defined by the substitution
$$a \mapsto a \qquad b \mapsto b^c \qquad c \mapsto c^b$$
and where 
$$w_1=[b,a][b,c][c,a][a,c]^b[a,b]^c[c,b]$$
$$w_2=[b,c]^a[c,b][b,a][c,a][a,b][a,c]^b$$
$$w_3=[c,b][a,b][b,c]^a[c,b]^2[b,a][b,c]^a[b,c]^a$$
$$w_4=[b,c]^a[a.b]^c[b,a]^2[a,c][a,b]^c[c,a][c,b].$$

\section{Properties of the Hanoi towers group and the proof of the main result}

In this section we compute the rigid kernel for $\Gamma$. For any branch group $G$, the rigid kernel is

\[
\ker(\widetilde{G} \rightarrow \overline{G})=\displaystyle\varprojlim_{n\geq 1} \St_G(n)/\Ri_G(n)
\]

\noindent where the maps $\rho_{n,n+k}: \St_G(n+k)/\Ri_G(n+k)\rightarrow \St_G(n)/\Ri_G(n)$ come from the natural inclusions 
${\St_G(n+k)\hookrightarrow \St_G(n)}$ and\\ ${\Ri_G(n+~k) \hookrightarrow \Ri_G(n)}.$  This is because, by definition, $\widetilde{G}$ is the subgroup of $\displaystyle\prod_{n\geq 1} G/\Ri_G(n)$ consisting of sequences $(g_n \Ri_G(n))_{n\geq 1}$ where 
\[g_{n+1}\Ri_G(n)=g_n\Ri_G(n)\]
for all $n$. Likewise, $\overline{G}$ is the subgroup of $\displaystyle\prod_{n\geq 1} G/\St_G(n)$ consisting of tuples $(h_n \St_G(n))_{n\geq 1}$ where 
\[h_{n+1}\St_G(n)=h_n\St_G(n)\] for all $n$. Thus the kernel of the map $\widetilde{G} \rightarrow \overline{G}$ is precisely those sequences $(g_n \Ri_G(n))_{n\geq 1}$ where for all $n$, $g_n\in \St_G(n)$, i.e. $\displaystyle\varprojlim_{n\geq 1} \St_G(n)/\Ri_G(n)$.

Note that the maps $\rho_{n,n+k}$ are far from being surjective and most of our work in computing the rigid kernel for $\Gamma$ will be in determining the image ${\rho_{n,n+k}(\St_\Gamma(n+~k)/\Ri_\Gamma(n+k))}$ for all $n$ and $k$.

First we observe that since each generator of $\Gamma$ has order 2, any element in $\Gamma$ can be expressed as a word in $a$, $b$, and $c$ using only the positive alphabet. Further, since each relator in (\ref{presentation}) can be written as a product of commutators, $\Gamma/\Gamma'\cong (C_2)^3$ where $C_2$ is a cyclic group of order 2. Thus a word in $a$, $b$, and $c$ is in $\Gamma'$ if and only if the sum of the exponents on each letter is congruent to 0 modulo 2.

Using the Reidemeister-Schreier method, we obtain a generating set for $\St_\Gamma(1)$: 
$$\alpha:=acab=(a,cb,a)_1 \qquad \beta:=abac=(a,a,bc)_1$$
$$\delta:=bcba=(ca,b,b)_1 \qquad \gamma:=babc=(b,b,ac)_1.$$

A group $G \leq $Aut$(\mathcal{T})$ is \textbf{level transitive} if it acts transitively on each level of Aut$(\mathcal{T})$. If $\mathcal{T}=X^*$, the infinite d-ary tree, then $G\leq \text{Aut}(X^*)$ is called \textbf{self-replicating} if $\St_G(u)@u=G$ for any vertex $u$. If $G$ is both self-replicating and acts transitively on the first level of the tree, then $G$ is level transitive.

As $\Gamma/\St_\Gamma(1)=S_3$, $\Gamma$ clearly acts transitively on the first level of the ternary tree. Thus to show it is level  transitive, it is sufficient to show it is self-replicating.

\begin{lemma}
\label{transrec}
$\Gamma$ is self-replicating.
\end{lemma}
\begin{proof}
From the generators obtained for $\St_\Gamma(1)$ above we see \\ ${\St_\Gamma(u)@u=\Gamma}$ for any vertex $u$ of level 1. Now suppose for any vertex $v$ of level $n$, $\St_\Gamma(v)@v=\Gamma$ and let $w$ be an immediate descendant of $v$. Then let $p$, $q$, $r$, and $s$ be the elements in $\St_\Gamma(v)$ that act as $\alpha$, $\beta$, $\delta$, and $\gamma$ on the subtree rooted at $v$. Then, $p$, $q$, $r$, and $s$ are in $\St_\Gamma(w)$ and $p@w$, $q@w$, $r@w$, and $s@w$ generate $\Gamma$. Thus, $\St_\Gamma(w)@w=\Gamma$. 
\end{proof}

$G \leq \text{Aut}(X^*)$ is said to be \textbf{self-similar} if $G@u$ is contained in $G$ for any $u \in V(X^*)$. 

For a vertex $u \in V(X^*)$ and for an element $g \in \text{Aut}(X^*)$, $u*g$ will be used to denote the automorphism of $X^*$ described by $(1,\dots,1, g, 1, \dots, 1)_n$ where $n=|u|$ and $g$ is in the $u$th coordinate; in other words $u*g$ acts as $g$ on the subtree rooted at $u$ and acts trivially outside this subtree. For a subgroup $K\leq \text{Aut}(X^*)$, $u*K:=\{u*k| k \in K\}$. Further, as $X^n$ represents the set of vertices of level $n$ in the tree, $X^n*K$ will be used to denote $\displaystyle\prod_{|u|=n} u* K$ where $K \leq \text{Aut}(X^*)$. A group $G\leq \text{Aut}(X^*)$ is said to be \textbf{regular branch} if it is level transitive and there is a subgroup $K$ with finite index in $G$ such that $u*K\leq K$ for all $u \in V(X^*)$ and such that $X^n*K$ has finite index in $G$ for all $n$. In this case, $K$ is called a \textbf{branching subgoup}. If a group is regular branch then it is also branch as $X^n*K\leq \Ri_G(n)$. 

An important observation that will be used frequently in the remainder of the paper is that if $X^n*H\leq G$, then

\begin{equation}
\label{observation}
\St_G(n+m)~\cap X^n*H~=X^n*\St_H(m).
\end{equation}

 This is because $X^n*H$ describes a disjoint action on each subtree rooted at the $n$th level, and so on each of these subtrees $\St_G(n+m)\cap X^n*H$ describes the collection of elements that are contained in $H$ and stabilize the $m$th level.

\begin{lemma}
\label{branchingsubgroup}
$\Gamma$ is a self-similar, regular branch group with branching subgroup $\Gamma'$.
\end{lemma}
\begin{proof}
The definition of the generators of $\Gamma$ easily implies that $\Gamma$ is self-similar. We will show by induction that $X^n * \Gamma'\leq \Gamma'$. For $n=1$, observe that 
$$(acbc)^2=(abab,1,1)_1=([a,b],1,1)_1$$
$$(abcb)^2=(acac,1,1)_1=([a,c],1,1)_1$$
$$c(baca)^2c=(bcbc,1,1)_1=([b,c],1,1)_1$$

\noindent and $(acbc)^2$, $(abcb)^2$, and $c(baca)^2c$ are all in $\Gamma'$ since $\Gamma/\Gamma'$ is an elementary abelian $2$-group.

From the description of the generators for $\St_\Gamma(1)$, we see that for all $g\in \Gamma$ there is an element $\tilde{g} \in \St_\Gamma(1)$ whose state in the first coordinate is $g$. Conjugating $(acbc)^2$ by $\tilde{g}$ produces the element $([a,b]^g, 1,1)_1$. Likewise, we can obtain the element that has any conjugate of $[a,c]$ or $[b,c]$ in the first coordinate and 1 in the second and third coordinates. As $\Gamma$ is transitive on all levels of $\mathcal{T}$, we obtain $X * \Gamma'\leq \Gamma'$.

Now assume for some $n\geq 1$, that $X^n * \Gamma' \leq \Gamma'$. By the base case, each copy of $\Gamma'$ contains a copy of $X*\Gamma'$. Therefore, $X^n*(X*\Gamma')\leq X^n * \Gamma'\leq \Gamma'$. But $X^n*(X*\Gamma')=X^{n+1}*\Gamma'$.
\end{proof}

\begin{lemma}
\label{rist}
For all $n\geq 1$, $\Ri_\Gamma(n)=X^n * \Gamma'$.
\end{lemma}
\begin{proof}
The proof is by induction on the level. By Lemma~\ref{branchingsubgroup}, ${X*~\Gamma'} \leq {\Ri_\Gamma(1)}\leq {\St_\Gamma(1)}\leq {X*\Gamma}$. Note that $(X*\Gamma)/(X*\Gamma') \cong (\Gamma/\Gamma')^3\cong [(\mathbb{Z}/2\mathbb{Z})^3]^3 \cong (\mathbb{Z}/2\mathbb{Z})^9$. Consider $H$, the rigid stabilizer of the first vertex of level $1$. The image H in $(\mathbb{Z}/2\mathbb{Z})^9$ is contained in the subspace $W$ consisting of vectors which have $0$ in the $i$th coordinate for $i\geq 4$. On the other hand, the image $U$ of $\St_\Gamma(1)$ in $(\mathbb{Z}/2\mathbb{Z})^9$ is spanned by the images of $\alpha$, $\beta$, $\delta$, and $\gamma$ which are
$$\widetilde{\alpha}=(1,0,0,0,1,1,1,0,0)$$
$$\widetilde{\beta}=(1,0,0,1,0,0,0,1,1)$$
$$\widetilde{\delta}=(1,0,1,0,1,0,0,1,0)$$
$$\widetilde{\gamma}=(0,1,0,0,1,0,1,0,1).$$

It is a simple exercise to see that $W\cap U=\{0\}$. It follows that $H\leq X*\Gamma'$ and thus $\Ri_\Gamma(1)=X*\Gamma'$.

Now assume for some $n\geq 1$ that $\Ri_\Gamma(n)=X^n*\Gamma'$. Then, again, by Lemma~\ref{branchingsubgroup}, $X^{n+1}*\Gamma' \leq \Ri_\Gamma(n+1)=\Ri_\Gamma(n+1)\cap X^n*\Gamma'=
X^n*\Ri_{\Gamma'}(1)\leq X^n*\Ri_\Gamma(1)=X^{n+1}*\Gamma'$, giving $X^{n+1}*\Gamma'=\Ri_\Gamma(n+~1)$.
\end{proof}

\begin{corollary}
\label{ristquotient}
For all $n$, $\Ri_\Gamma(n)\St_\Gamma(n+1)/\St_\Gamma(n+1)=(A_3)^{3^n}$ where $A_3$ is the alternating group on 3 letters.
\end{corollary}
\begin{proof}
The projection $\Gamma\rightarrow \Gamma/\St_\Gamma(1)\cong S_3$ takes $\Gamma'$ onto $A_3$, hence $\Gamma'/\St_{\Gamma'}(1)\cong A_3$. Further,
\[
{\Ri_\Gamma(n)\St_\Gamma(n+1)/\St_\Gamma(n+1)}\cong {X^n*\Gamma'/[(X^n*\Gamma')\cap {\St_\Gamma(n+1)]}}
\]
\[
\cong {X^n*\Gamma'/X^n*\St_{\Gamma'}(1)}\cong {(\Gamma'/\St_{\Gamma'}(1))^{3^n}}\cong {(A_3)^{3^n}}. 
\]
\end{proof}

\begin{corollary}
\label{elab}
The rigid kernel for $\Gamma$ is an elementary abelian $2$-group. 
\end{corollary}
\begin{proof}
Since $\St_{\Gamma}(n)\leq X^n*\Gamma$, we have 
\[{\St_\Gamma(n)/\Ri_\Gamma(n)}={\St_\Gamma(n)/X^n*\Gamma'}\]
is a subspace of $X^n*\Gamma/X^n*\Gamma'\cong (\Gamma/\Gamma')^{3^n}$ which is an elementary abelian $2$-group. An inverse limit of elementary abelian $2$-groups is an elementary abelian $2$-group.
\end{proof}

\begin{corollary}
\label{index}
$|\St_\Gamma(1)/\Ri_\Gamma(1)|=16$  and $|\St_{\Gamma'}(1)/\Ri_{\Gamma'}(1)|=4$.
\end{corollary}
\begin{proof}
We have seen in the proof of Lemma \ref{rist} that $\St_{\Gamma}(1)/\Ri_\Gamma(1)=\St_{\Gamma}(1)/X*\Gamma'=U$ is a four dimensional vector space over $\mathbb{F}_2$ (the images of $\alpha$, $\beta$, $\delta$, and $\gamma$ form a basis). Hence $U$ has $16$ elements. 

Now, by Lemmas \ref{branchingsubgroup} and \ref{rist}, we see that $\Ri_{\Gamma}(n)=\Ri_{\Gamma'}(n)=X^n*\Gamma'$. This gives $\St_{\Gamma'}(1)/\Ri_{\Gamma'}(1)=(\St_\Gamma(1) \cap \Gamma')/\Ri_\Gamma(1)={U\cap (\Gamma'/\Ri_\Gamma(1)})$. Moreover, since a word in $a$, $b$, and $c$ is in $\Gamma'$ if and only if each generator appears in it an even number of times, a word in $\alpha$, $\beta$, $\delta$, and $\gamma$ is in $\Gamma'$ if and only if the number of appearances of $\alpha$ and $\beta$ have the same parity and the number of appearances of $\delta$ and $\gamma$ have the same parity. It follows that $U\cap (\Gamma'/\Ri_\Gamma(1))$ is the two dimensional subspace spanned by $\widetilde{\alpha}+\widetilde{\beta}$ and $\widetilde{\delta} + \widetilde{\gamma}$.
\end{proof}

As $\Gamma$ is self-replicating,  if $g \in \St_\Gamma(u)$, then $g@u$ must also be an element of $\Gamma$. Corollary \ref{index} and the following lemma serve to elucidate the action of $\Gamma$ on the top levels of $\mathcal{T}$.

\begin{lemma}
\label{stabilizer12}
\begin{enumerate}
\item $\Gamma/\St_\Gamma(1)\cong S_3$, the symmetric group on three letters. 
\item $\St_\Gamma(1)/\St_\Gamma(2)$ considered as a subgroup of $(S_3)^3$ is the kernel of the homomorphism $\phi: (S_3)^3 \rightarrow C_2$ where $\phi$ sums the signs of the permutation in each coordinate. This quotient has order $2^2\cdot 3^3$.
\end{enumerate}
\end{lemma}

\begin{proof}
a) We have already observed that $\Gamma/\St_\Gamma(1)\cong S_3$.\\

\noindent b) The images of $\alpha$, $\beta$, $\delta$, and $\gamma$ in $\St_\Gamma(1)/\St_\Gamma(2)$ are 
$$\overline{\alpha}=(\sigma_{23},\sigma_{123},\sigma_{23}) \qquad \overline{\beta}=(\sigma_{23},\sigma_{23},\sigma_{132})$$
$$\overline{\delta}=(\sigma_{123},\sigma_{13},\sigma_{13}) \qquad \overline{\gamma}=(\sigma_{13},\sigma_{13},\sigma_{132}).$$
Thus $\overline{\alpha},\overline{\beta},\overline{\delta},$ and $\overline{\gamma}$ are in $\ker(\phi)$.
Further $\overline{\delta}^2=(\sigma_{132},1,1)$ and, by spherical transitivity, this implies that $(A_3)^3\leq \St_\Gamma(1)/\St_\Gamma(2)$. Also, $\overline{\alpha\beta}=(1, \sigma_{13}, \sigma_{13})$ and $\overline{\delta\gamma}=(\sigma_{23},1, \sigma_{23})$. Collectively, these elements generate $\ker(\phi)$.
\end{proof}

Now we apply our knowledge of the permutations appearing on the top levels of the tree to gain an understanding of action on subtrees rooted at the lower levels.

\begin{lemma}
\label{stabilizerquotients}
For $n\geq 1$, we have isomorphisms ${\St_{\Gamma'}(n)/\St_{\Gamma'}(n+1)}\cong { \St_\Gamma(n)/\St_\Gamma(n+1)\cong X^{n-1}*\St_\Gamma(1)/X^{n-1}*\St_\Gamma(2)}$. In particular, all three groups have order $2^{2\cdot 3^{n-1}}\cdot 3^{3^n}$.
\end{lemma}
\begin{proof}
Since \[{\St_{\Gamma'}(n)/\St_{\Gamma'}(n+1)=\St_{\Gamma'}(n)/(\St_{\Gamma'}(n)\cap \St_\Gamma(n+1))},\]
 the group ${\St_{\Gamma'}(n)/\St_{\Gamma'}(n+1)}$ can be considered as a subgroup of \\ ${\St_\Gamma(n)/\St_\Gamma(n+1)}$. By self-similarity, $\St_\Gamma(n)/\St_\Gamma(n+1)$ can be considered as a subgroup of $(X^{n-1}*\St_\Gamma(1))/(X^{n-1}*\St_\Gamma(2))$, a group of order $2^{2\cdot 3^{n-1}}\cdot 3^{3^n}$. Therefore it suffices to prove that 
\[|\St_{\Gamma'}(n)/\St_{\Gamma'}(n+1)|\geq 2^{2\cdot 3^{n-1}}\cdot 3^{3^n}.\]

Observe that $\Gamma'/\St_{\Gamma'}(1)\cong A_3$, generated by the image of $[a,b]=(ab,a,b)\sigma_{123}$ and recall that $\Ri_\Gamma(n)=X^n*\Gamma'\leq \Gamma'$. It follows that $\St_{\Gamma'}(n)/\St_{\Gamma'}(n+1)$ contains $(X^n*\Gamma')\St_{\Gamma'}(n+1)/\St_{\Gamma'}(n+1)$. Note that
\[
(X^n*\Gamma') \St_{\Gamma'}(n+1)/\St_{\Gamma'}(n+1) \cong X^n*\Gamma'/(\St_{\Gamma'}(n+1)\cap X^n*\Gamma')
\]
\[
=X^n*\Gamma'/X^n*\St_{\Gamma'}(1) \cong (\Gamma'/\St_{\Gamma'}(1))^{3^n} \cong (A_3)^{3^n}.
\]

\noindent Therefore, $\St_{\Gamma'}(n)/\St_{\Gamma'}(n+1)$ has a subgroup of order $3^{3^n}$.

Now,  $\St_{\Gamma'}(n)/\St_{\Gamma'}(n+1)$ also contains a subgroup isomorphic to $(X^{n-1}*\Gamma'\cap \St_{\Gamma'}(n))/(X^{n-1}*\Gamma'\cap\St_{\Gamma'}(n+1))$. Moreover, by (\ref{observation}) this subgroup is isomorphic to $(\St_{\Gamma'}(1)/\St_{\Gamma'}(2))^{3^{n-1}}$ which has order $2^{2\cdot 3^{n-1}}$ by Corollary \ref{index}.
\end{proof}

Now, we have all the tools needed to prove the main theorem.

\begin{theorem}
The rigid kernel $\ker(\widetilde{\Gamma}\rightarrow \overline{\Gamma})$ is the Klein four-group.
\end{theorem}
\begin{proof}
By Corollary \ref{elab}, the rigid kernel is an elementary abelian $2$-group, so we only need to show that it has order 4.

For notational simplicity, for all $n\geq 1$, define $\Gamma_{n}:=\St_\Gamma(n)/\Ri_\Gamma(n)$. Further, under the natural map from $\Gamma_{n+k}$ to $\Gamma_{n}$, let $H_{n,n+k}$ be the image of $\Gamma_{n+k}$ in $\Gamma_n$, let $K_{n,n+k}$ be the kernel of this map, and let $Q_{n, n+k}$ be the cokernel of this map (note that $H_{n,n+k}\unlhd \Gamma_n$).  

Recall that the rigid kernel is $\displaystyle\varprojlim_{n\geq 1} \Gamma_n$. We will show that for all $n$,  $H_{n,n+1}=H_{n, n+2}$ and that both have order $4$. This implies that for each $n$, the maps $H_{n+1, n+2} \rightarrow H_{n,n+1}$ are isomorphisms and hence $\displaystyle\varprojlim_{n\geq 1} \Gamma_n=\displaystyle\varprojlim_{n\geq 1} H_{n,n+1}$ also has order $4$, completing the proof. 

The first step in doing this is to determine $H_{n,n+1}$. We have the exact sequence
\begin{align}
 \label{exact}
1 \rightarrow K_{n, n+1} \rightarrow \Gamma_{n+1} \rightarrow \Gamma_n \rightarrow Q_{n,n+1} \rightarrow 1.
\end{align}

Now 
\[
K_{n,n+1} = (\St_\Gamma(n+1)\cap\Ri_\Gamma(n))/\Ri_\Gamma(n+1)
\]
\[  
  = (\St_\Gamma(n+1)\cap X^n*\Gamma')/ X^{n+1}*\Gamma'= X^n*\St_{\Gamma'}(1)/X^{n+1}*\Gamma'
\]
\[
  \cong (\St_{\Gamma'}(1)/\Ri_{\Gamma'}(1))^{3^n},
\]

\noindent hence $|K_{n,n+1}|=2^{2\cdot 3^n}$ from Corollary \ref{index}. 

Also, $Q_{n,n+1}=\St_\Gamma(n)/\Ri_\Gamma(n)\St_\Gamma(n+1)$ has $2^{2\cdot 3^{n-1}}$ elements by Lemma \ref{stabilizerquotients} and Corollary \ref{ristquotient}.

Since sequence~(\ref{exact}) is exact, $\frac{|\Gamma_{n+1}|}{|\Gamma_n|}=\frac{|K_{n,n+1}|}{|Q_{n,n+1}|}=2^{4\cdot3^{n-1}}$. Further, by Corollary \ref{index}, $|\Gamma_1|=16$. Collectively, we obtain
$$|\Gamma_n|=2^4 \prod_{i=2}^n 2^{4\cdot 3^{i-2}}=2^{2\cdot (3^{n-1}+1)}$$
and the size of $H_{n,n+1}$ is $\frac{|\Gamma_{n+1}|}{|K_{n,n+1}|}=\frac{2^{2(3^n+1)}}{2^{2\cdot 3^n}}=4$.

Now it remains to show that $Q_{n,n+2}=Q_{n,n+1}$ as this would imply $H_{n,n+2}=~H_{n,n+1}$ and moreover that $H_{n+1,n+2}$ maps isomorphically to $H_{n,n+1}$ for all $n$.

Now 
$$Q_{n,n+i}=\St_\Gamma(n)/(X^n*\Gamma')\St_\Gamma(n+i).$$

Thus showing $Q_{n,n+1}=Q_{n,n+2}$ is the same as showing 
\[(X^n*\Gamma')\St_\Gamma(n+~1)= (X^n*\Gamma')\St_\Gamma(n+2).\]

By Lemma \ref{stabilizerquotients}, 
\[
{\St_\Gamma(n+1)/\St_\Gamma(n+2)}\cong {X^n*\St_\Gamma(1)/X^n*\St_\Gamma(2)}
\]
\[
\cong {X^n*\St_{\Gamma'}(1)/X^n*\St_{\Gamma'}(2)}.
\]
Hence, ${\St_\Gamma(n+1)}={\St_\Gamma(n+2)(X^n*\St_{\Gamma'}(1))}$ and we obtain 
\[
{(X^n*\Gamma')\St_\Gamma(n+1)}={(X^n*\Gamma')(X^n*\St_{\Gamma'}(1))\St_\Gamma(n+2)}
\]
\[
={(X^n*\Gamma')\St_\Gamma(n+2)}.
\]
\end{proof}

\bibliographystyle{alpha}

\end{document}